\documentclass{amsart}
\usepackage{amssymb}

\usepackage{stmaryrd}
\usepackage{mathrsfs}

\newtheorem{theorem}{Theorem}[section]
\newtheorem{lemma}[theorem]{Lemma}

\newtheorem{prop}[theorem]{Proposition}

\theoremstyle{definition}

\theoremstyle{remark}

\numberwithin{equation}{section}

\newcommand{\FF}{{\mathbb{F}}}

\newcommand{\PSL}{{\operatorname{PSL}}}
\newcommand{\ASL}{{\operatorname{ASL}}}
\newcommand{\PGaL}{{\operatorname{P\Gamma L}}}

\newcommand{\PGL}{{\operatorname{PGL}}}

\newcommand{\Alt}{{\operatorname{Alt}}}
\newcommand{\Sym}{{\operatorname{Sym}}}

\begin{document}

\title{Permutation Groups and Orbits on Power Sets}

\author{Yong Yang}
\address{Department of Mathematics, Texas State University, 601 University Drive, San Marcos, TX 78666, USA.}
\makeatletter
\email{yang@txstate.edu}
\makeatother

\date{}



\begin{abstract}
Let $G$ be a permutation group of degree $n$ and let $s(G)$ denote the number of set-orbits of $G$. We determine $\inf(\frac {\log_2 s(G)} n)$ over all groups $G$ that satisfy certain restrictions on composition factors (i.e. $\Alt(k), k > 4$ cannot be obtained as a quotient of a subgroup of $G$).
\end{abstract}

\maketitle
\section{Introduction} \label{sec:introduction8}
A permutation group $G$ acting on a set $\Omega$ induces a permutation group on the power set $\mathscr{P}(\Omega)$. We call the orbits of this action set-orbits. Let $s(G)$ denote the number of set-orbits of $G$. Obviously $s(G)\geq |\Omega| + 1$ as sets of different cardinality belong to different orbits. 

In a paper of Babai and Pyber ~\cite[Theorem 1]{BAPYBER}, they show that if G has no large alternating composition factors then $s(G)$ is exponential in $n$. More precisely they prove the following result ~\cite[Theorem 1]{BAPYBER}. Let $G$ be a permutation group of degree $n$ and let $s(G)$ denote the number of set-orbits of $G$.  Assume that $G$ does not contain any alternating group $\Alt(k), k>t (t \geq 4)$ as a composition factor, then $\frac {\log_2 s(G)} n \geq \frac {C_1} t$ for some positive constant $C_1$.

In the same paper, they raise the following question, what is $\inf(\frac {\log_2 s(G)} n)$ over all solvable groups $G$? This question is answered in a recent paper of the author in ~\cite{YY11}. Clearly, a more interesting question is to answer the following question, what is $\inf(\frac {\log_2 s(G)} n)$ over all groups $G$ that does not contain any alternating group $\Alt(k), k>4$ as a composition factor?

We study this question in this paper. It turns out that $\inf(\frac {\log_2 s(G)} n)= \lim_{k \mapsto \infty} \frac {\log_2 s(M_{24} \wr M_{12} \wr S_4 \wr \cdots \wr S_4)} {24 \cdot 12 \cdot 4^k}$. Let $a_k= \frac {\log_2 s(M_{24} \wr M_{12} \wr S_4 \wr \cdots \wr S_4)} {24 \cdot 12 \cdot 4^k}$, we know that this sequence is decreasing and \

\[\lim_{k \mapsto \infty} a_k \approx 0.1712268716679245433. \]

Clearly, the main difficulty of this work is to identify the group that achieves the minimum bound.

\section{Notation and Lemmas} \label{sec:Notation and Lemmas}

Notation:
\begin{enumerate}
\item We use $H \wr S$ to denote the wreath product of $H$ with $S$ where $H$ is a group and $S$ is a permutation group.
\item Let $G$ be a solvable permutation group of degree $n$ and we use $s(G)$ to denote the number of set-orbits of $G$ and we denote $rs(G)=\frac {\log_2 s(G)} n$.\\
\end{enumerate}

We recall some basic facts about the decompositions of transitive groups. Let $G$ be a transitive permutation group acting on a set $\Omega$, $|\Omega| = n$. A system of imprimitivity is a partition of $\Omega$, invariant under $G$. A primitive group has no non-trivial system of imprimitivity. Let $(\Omega_1, \cdots, \Omega_m)$ denote a system of imprimitivity of $G$ with maximal block-size $b$ $(1 \leq b <n; b = 1$ if and only if $G$ is primitive; $bm=n$).

Let $H$ denote the normal subgroup of $G$ stabilizing each of the blocks $\Omega_i$. Then $G/H$ is a primitive group of degree $m$ acting upon the set of blocks $\Omega_i$. If $G_i$ denotes the permutation group of degree $b$ induced on $\Omega_i$ by the setwise stabilizer of $\Omega_i$ in $G$, then the groups $G_i$ are permutationally equivalent transitive groups and $H \leq G_1 \times \cdots \times G_m < \Sym(\Omega)$.

Let $G$ be a transitive group of degree $n$ and assume that $G$ is not primitive. Let us consider a system of imprimitivity of $G$ that consists of $m \geq 2$ blocks of size $b$, $b$ maximal. Thus $G \lesssim H \wr P_1$ where $P_1$ is the primitive quotient group of $G$ that acts upon the $m$ blocks. We may keep doing this, eventually we may view $G \lesssim K \wr P_j \cdots \wr P_1$ where $K$ is a primitive group and $P_i$ are all primitive groups. We say that $G$ is induced from the primitive group $K$.\\

\begin{lemma}  \label{lem1}
If $H \leq G \leq \Sym(\Omega)$, then $s(G) \leq s(H) \leq s(G) \cdot |G:H|$.
\end{lemma}

\begin{lemma}  \label{lem2}
Assume $G$ is intransitive on $\Omega$ and has orbits $\Omega_1, \cdots , \Omega_m$. Let $G_i$ be the restriction of $G$ to $\Omega_i$. Then \[s(G) \geq s(G_1) \cdot \cdots \cdot s(G_m).\]
\end{lemma}
\begin{proof}
$G \leq G_1 \times \cdots \times G_m$ so we can apply Lemma ~\ref{lem1}. Clearly $s(G_1 \times \cdots \times G_m)=s(G_1) \cdot \cdots \cdot s(G_m)$.
\end{proof}

\begin{lemma}  \label{induction}
Let $(\Omega_1, \cdots, \Omega_m)$ denote a system of imprimitivity of $G$ with maximal block-size $b$ $(1 \leq b <n; b = 1$ if and only if $G$ is primitive; $bm=n$). Let $H$ denote the normal subgroup of $G$ stabilizing each of the blocks $\Omega_i$. Denote $s=s(G_i)$. Then
\begin{enumerate}
\item $s(G) \geq s^m/|G/H|$.
\item $s(G) \geq {{s+m-1} \choose s-1}$. Also, the equation holds if $G/H \cong S_m$.
\end{enumerate}
\end{lemma}
\begin{proof}

Let $A$ be a subset of $\Omega$ and let $\alpha_j(0 \leq j \leq b)$ denote the number of $s$ orbit-element sets among the $A \cap \Omega_i$. Let $B$ be another subset of $\Omega$ with the number $\beta_j$ defined similarly. If $A$ and $B$ are in the same orbit of $G$ then $\alpha_j=\beta_j$ for $(0 \leq j \leq b)$. Therefore $s(G)$ is at least the number of partitions of $m$ into $s$ non-negative integers (where the order of the summands is taken into consideration). It is well known that this number is \[{{s+m-1} \choose s-1}\] which proves (2).
\end{proof}


We need the following estimates of the order of the primitive permutation groups.

\begin{theorem} \label{thm1}
  Let $G$ be a primitive permutation group of degree $n$ where $G$ does not contain $A_n$. Then
  \begin{enumerate}
  \item $|G| < 50  \cdot n^{\sqrt n}$.
  \item $|G| < 3^n$. Moreover, if $n> 24$, then $|G| < 2^n$.
  \item $|G| \leq 2^{0.76 n}$ when $n \geq 25$ and $n \neq 32$.
  \end{enumerate}
\end{theorem}
\begin{proof}
(1) is ~\cite[Corollary 1.1(ii)]{Maroti} and (2) is ~\cite[Corollary 1.2]{Maroti}.

(3) follows from (1) for $n \geq 89$, and one may check the results using GAP \cite{GAP} for $25 \leq n \leq 88$.

\end{proof}

\begin{lemma} \label{lem3}
  Let $G$ be a primitive permutation group of degree $n$ where $G$ does not contain any alternating group $\Alt(k), k>4$ as a composition factor.
  \begin{enumerate}
   \item $s(G)\geq 35$ when $n=14$.
   \item $s(G)\geq 46$ when $n=15$.
   \item $s(G)\geq 32$ when $n=16$.
   \item $s(G)\geq 48$ when $n=17$.
   \item $s(G)\geq 158$ when $n=21$.
   \item $s(G)\geq 105$ when $n=22$.
   \item $s(G)\geq 72$ when $n=23$.
   \item $s(G)\geq 49$ when $n=24$.
   \item $s(G)\geq 361$ when $n=32$.

  \end{enumerate}
\end{lemma}
\begin{proof}
The results are checked by GAP ~\cite{GAP}.
\item $n=14$, the bound $35$ is attained by PrimitiveGroup(14,2) $\cong \PGL(2,13)$.
\item $n=15$, the bound $46$ is attained by PrimitiveGroup(15,4) $\cong \PSL(4,2)$.
\item $n=16$, the bound $32$ is attained by PrimitiveGroup(16,11) $\cong 2^4.\PSL(4,2)$.
\item $n=17$, the bound $48$ is attained by PrimitiveGroup(17,8)  $\cong \PGaL(2,2^4)$.
\item $n=21$, the bound $158$ is attained by PrimitiveGroup(21,7) $\cong \PGaL(3,4)$.
\item $n=22$. There are only two groups. If $G \cong M_{22}$, then $s(G)=130$. If $G \cong M_{22}.2$, then $s(G)=105$.
\item $n=23$, the group with the second largest order has order $506$ and $s(G) \geq 16578$. The group with the largest order is $G \cong M_{23}$, and $s(G)=72$.
\item $n=24$, the group with the second largest order has order $12144$ and $s(G) \geq 1382$.  The group with the largest order is $G \cong M_{24}$, and $s(G)=49$.
\item $n=32$, the group with the second largest order has order $29760$ and $s(G) \geq 144321$.  The group with the largest order is $G \cong $ PrimitiveGroup(32,3) $\cong \ASL(5,2)$, and $s(G) \geq 361$. We remark here that $361$ is a lower bound estimated by GAP ~\cite{GAP} using random search, we cannot get the exact value though.
\end{proof}

\section{Main Theorems} \label{sec:maintheorem}
We define a sequence $\{s_k\}_{k \geq 0}$ where $s_{-1}=s(M_{24})=49$, $s_0=s(M_{24} \wr M_{12})=$ and $s_{k+1}={{s_{k}+3} \choose 4}$ for $k \geq 0$.

Clearly the sequence $\{s_k\}_{k \geq 0}$ is strictly increasing.

We define a sequence $\{a_k\}_{k \geq 0}$ where $a_k=\frac {\log_2{s_k}} {24 \cdot 12 \cdot 4^k}$.

\[a_{k+1}=\frac {\log_2{{{s_{k}+3} \choose 4}}} {24 \cdot 12 \cdot 4^{k+1}}=\frac {\log_2 (\frac {(s_{k}+3)(s_{k}+2)(s_{k}+1)(s_{k})} {24})} {24 \cdot 12 \cdot 4^{k+1}} < \frac {\log_2 (s_k)^4} {24 \cdot 12 \cdot 4^{k+1}}=a_k.\]

Thus we know that the sequence $\{a_k\}_{k \geq 0}$ is strictly decreasing. Since $a_k>0$, $\lim_{k \mapsto \infty} a_k$ exists.

In order to provide a good estimate of the value of  $\lim_{k \mapsto \infty} a_k$, one needs to calculate the exact value of $s(M_{24} \wr M_{12})$. Using GAP ~\cite{GAP}, one may easily obtain that $s(M_{24})=49$. On the other hand, it is not easy to calculate the exact value of $s(M_{24} \wr M_{12})$ using GAP ~\cite{GAP} directly due to the calculation complexity of the orbit-stabilizer algorithm.

In order to calculate $s(M_{24} \wr M_{12})$, we need to break down the calculation into a few manageable steps. We first fix some notation. Let $\Pi$ be the set of all the composition of $12$ and let $\pi$ be a partition of $12$. Here we use $B(\pi)$ to denote the number of the blocks of the partition. We know that $B(\pi)=n_1 +n_2 \cdots +n_k$, where $n_1$ is the number of blocks of the largest size, $n_2$ is the number of blocks of the second largest size and so on, and we define $F(\pi)=n_1! \cdot n_2! \cdots n_k!$.

We define $N(\pi)$ to be the number of orbits of $M_{12}$ on all the multiset permutations (permutations with repetition) of a set of $12$ elements with the partition $\pi$ as multiset structure.



One can read the value of $N(\pi)$ from Table 1 and Table 2. The results in Table 1 and Table 2 are obtained using computer program GAP ~\cite{GAP}.

For example, for the partition $\pi=(3,3,2,1,1,1,1)$, we have that $B(\pi)=7$, $F(\pi)=2! \cdot 4!$ and $N(\pi)=70$.

$s(M_{24} \wr M_{12})$ can be calculated using the following formula.

\[s(M_{24} \wr M_{12})=\sum_{\pi \in \Pi} N(\pi) \cdot \frac {_{49} P_{B(\pi)}}{F(\pi)}\]

Using the previous formula and Table $1, 2$, we obtain that $s_0=s(M_{24} \wr M_{12}) = 2,017,737,434,447,329$.

By some calculations, one gets the following. $a_0 \approx 0.1765335412289444$, $a_1 \approx  0.172553539058179$ and $a_2 \approx  0.171558538515488$. 

\begin{lemma} \label{lem4}
  Let $G$ be a primitive permutation group of degree $n$ where $G$ does not contain any alternating group $\Alt(k), k>4$ as a composition factor. Then $rs(G) \geq a_2$.
\end{lemma}
\begin{proof}
Suppose that $n \geq 25$ and $n \neq 32$, $s(G) \geq 2^{0.24 n}$ by Theorem ~\ref{thm1}(3) and $rs(G) \geq 0.24 > a_2$.

Suppose that $n \leq 24$, then $s(G) \geq n+1$ and $rs(G) > 0.199 > a_2$.

Suppose that $n =32$, then $s(G) \geq 361$ by Lemma ~\ref{lem3}(9) and $rs(G) > 0.26 > a_2$.
\end{proof}

Next we prove a general induction theorem.
\begin{theorem} \label{thm2}
  Let $G$ be a permutation group of degree $n$ induced from $H$ where $H$ is a permutation group of degree $m$. Let $\alpha=24^{1/3}$. Suppose that \[\frac {\log_2 s(H)} m - \frac {\log_2(\alpha)} m \geq \beta, \] then $\frac {\log_2 s(G)} n \geq \beta$.
\end{theorem}
\begin{proof}
We may assume that $G \lesssim H \wr P_1 \cdots \wr P_j$, where $P_1$ is a primitive permutation group of degree $k$ and all the $P_i$s are primitive permutation groups. Then we know that $|P_1| \leq \alpha ^{k-1}$ by Theorem ~\ref{thm1}(1).

Since \[\frac {\log_2 s(H)} m - \frac {\log_2(\alpha)} m \geq \beta, \] we have that $s(H) \geq \alpha 2^{m \beta}$.

By Lemma ~\ref{induction}(1) we have $s(H \wr P_1) \geq s(H)^k/|P_1| \geq \alpha^k 2^{mk \beta}/\alpha^{k-1}=\alpha 2^{mk \beta}$.
Thus we have
 \[\frac {\log_2 s(H \wr P_1)} {mk} - \frac {\log_2(\alpha)} {mk} \geq \beta.\]
Now the result follows by induction.
\end{proof}

\begin{prop} \label{prop1}
  Let $G$ be a permutation group of degree $n$ where $G$ does not contain any alternating group $\Alt(k), k>4$ as a composition factor. Let $G$ be induced from $H$ where $H$ be a primitive permutation group of degree $m$. If $H \not\cong M_{24}$, then $\frac {\log_2 s(G)} n \geq a_2$.
\end{prop}
\begin{proof}
In view of Theorem ~\ref{thm2}, it suffices to check that \[\frac {\log_2 s(H)} m - \frac {\log_2(\alpha)} m \geq a_2.\] We call this inequality $\star$.

Suppose that $m \geq 25$ and $m \neq 32$, then $|H| \leq 2^{0.76 m}$ by Theorem ~\ref{thm1}(2). Since $s(H) \geq 2^m/|H|$, we have $\frac {\log_2 s(H)}  m \geq 0.24$. It is easy to check that $\star$ is satisfied.

Suppose that $m=32$, then $s(H) \geq 361$ by Lemma ~\ref{lem3}(9), it is easy to check that $\star$ is satisfied.

Suppose that $m=24$ and $H \not\cong M_{24}$, then $s(H) \geq 1382$ by Lemma ~\ref{lem3}(8), it is easy to check that $\star$ is satisfied.

Suppose that $m=23$, then $s(H) \geq 72$ by Lemma ~\ref{lem3}(7), it is easy to check that $\star$ is satisfied.

Suppose that $m=22$, then $s(H) \geq 105$ by Lemma ~\ref{lem3}(6), it is easy to check that $\star$ is satisfied.

Suppose that $m=21$, then $s(H) \geq 158$ by Lemma ~\ref{lem3}(5), it is easy to check that $\star$ is satisfied.

Suppose that $m=20$, then $|H| \leq 6840$ by Table 3 and $s(H) \geq 2^{20}/6840$, it is easy to check that $\star$ is satisfied.

Suppose that $m=19$, then $|H| \leq 342$ by Table 3 and $s(H) \geq 2^{19}/342$, it is easy to check that $\star$ is satisfied.

Suppose that $m=18$, then $|H| \leq 4896$ by Table 3 and $s(H) \geq 2^{18}/4896$, it is easy to check that $\star$ is satisfied.

Suppose that $m=17$, then $s(H) \geq 48$ by Lemma ~\ref{lem3}(4), it is easy to check that $\star$ is satisfied.

Suppose that $m=16$, then $s(H) \geq 32$ by Lemma ~\ref{lem3}(3), it is easy to check that $\star$ is satisfied.

Suppose that $m=15$, then $s(H) \geq 46$ by Lemma ~\ref{lem3}(2), it is easy to check that $\star$ is satisfied.

Suppose that $m=14$, then $s(H) \geq 35$ by Lemma ~\ref{lem3}(1), it is easy to check that $\star$ is satisfied.

Suppose that $m=13$, then $s(H) \geq 14$. It is easy to check that $\star$ is satisfied.

Suppose that $m=12$, then $s(H) \geq 13$. It is easy to check that $\star$ is satisfied.

Suppose that $m=11$, then $s(H) \geq 12$. It is easy to check that $\star$ is satisfied.

Suppose that $m=10$, then $s(H) \geq 11$. It is easy to check that $\star$ is satisfied.

Suppose that $m=9$, then $s(H) \geq 10$. It is easy to check that $\star$ is satisfied.

Suppose that $m=8$, then $s(H) \geq 9$. It is easy to check that $\star$ is satisfied.

Suppose that $m=7$, then $s(H) \geq 8$. It is easy to check that $\star$ is satisfied.

Suppose that $m=5$, then $s(H) \geq 6$. It is easy to check that $\star$ is satisfied.

Suppose that $m=4$, then $s(H) \geq 5$. It is easy to check that $\star$ is satisfied.

By Lemma ~\ref{lem4}, we may assume that $G \lesssim H \wr P_1 \cdots \wr P_j$ where $\deg(P_1)=m_1>0$. Let $K=H \wr P_1$.

In view of Theorem ~\ref{thm2}, it suffices to check that \[\frac {\log_2 s(K)} {m m_1} - \frac {\log_2(\alpha)} {m m_1} \geq a_2.\] We call this inequality $\star (K)$.

Suppose that $m=3$, then $s(H) \geq 4$.

$s(K)=s(H)^{m_1}/|P_1|=4^{m_1}/|P_1|$ by Lemma ~\ref{induction}(1). If $m_1 \geq 25$ and $m_1 \neq 32$, then $|P_1| \leq 2^{0.76 m_1}$ by Theorem ~\ref{thm1}(3) and it is easy to check that $\star(K)$ is satisfied.

If $17 \leq m_1 \leq 24$ or $m_1=32$, we use the results in Table 3 to estimate $|P_1|$ and it is easy to check that $\star(K)$ is satisfied.

If $m_1 \leq 16$, then we have that $s(K) \geq {{3+m_1} \choose 3}$ by Lemma ~\ref{induction}(2) and it is easy to check that $\star(K)$ is satisfied.

Suppose that $m=2$, then $s(H) \geq 3$.

$s(K)=s(H)^{m_1}/|P_1|=3^{m_1}/|P_1|$ by Lemma ~\ref{induction}(1). If $m_1 \geq 25$ and $m_1 \neq 32$, then $|P_1| \leq 2^{0.76 m_1}$ by Theorem ~\ref{thm1}(3) and it is easy to check that $\star(K)$ is satisfied.

If $18 \leq m_1 \leq 24$ or $m_1=32$, we use the results in Table 3 to estimate $|P_1|$ and it is easy to check that $\star(K)$ is satisfied.

If $m_1 \leq 17$, then we have that $s(K) \geq {{2+m_1} \choose 2}$ by Lemma ~\ref{induction}(2) and it is easy to check that $\star(K)$ is satisfied.
\end{proof}

\begin{prop} \label{prop2}
  Let $G$ be a permutation group of degree $n$ where $G$ does not contain any alternating group $\Alt(k), k>4$ as a composition factor. Let $G$ be induced from $H$ where $H$ be a primitive permutation group of degree $24$ and $H \cong M_{24}$. Thus $G \lesssim H \wr P_1 \cdots \wr P_j$. If $P_1 \not \cong M_{12}$, then $\frac {\log_2 s(G)} n \geq a_2$.
\end{prop}
\begin{proof}
Clearly $s(H)=49$ by Lemma ~\ref{lem3}(8). Let $K=H \wr P_1$ and $m_1=\deg(P_1)$. In view of Theorem ~\ref{thm2}, it suffices to check that \[\frac {\log_2 s(K)} {24 m_1} - \frac {\log_2(\alpha)} {24 m_1} \geq a_2.\]

If $32 \geq m_1 \geq 5$, then $s(K) \geq 49^{m_1}/|P_1|$ by Lemma ~\ref{induction}(1). We may estimate $|P_1|$ using Table 3, and it is easy to check that the previous inequality is satisfied.

If $m_1 \geq 33$, then $s(K) \geq 49^{m_1}/|P_1|$ by Lemma ~\ref{induction}(1). We may estimate $|P_1|$ using Theorem ~\ref{thm1}(3), and it is easy to check that the previous inequality is satisfied.

If $m_1 \leq 4$, then $s(K) \geq {{48+m_1} \choose 48}$ by Lemma ~\ref{induction}(2), and it is easy to check that the previous inequality is satisfied.
\end{proof}





\begin{theorem} \label{thm3}
  Let $H$ be a permutation group of degree $24 \cdot 12 \cdot 4^k$ where $k \geq 0$ and $H \cong M_{24} \wr M_{12} \wr S_4 \wr \cdots \wr S_4$. Then $rs (H \wr S_4 \wr S_4 \wr S_4 \wr S_4) \leq rs(H \wr P_1 \cdots \wr P_j)$ if $\deg (P_1) \neq 4$.
\end{theorem}
\begin{proof}
Let $A=s(H)$ and $n=\deg(H)=24 \cdot 12 \cdot 4^k$.

$B=s(H \wr S_4) = {{A+3} \choose {A-1}}= {{A+3} \choose 4}$ by Lemma ~\ref{induction}(2).

$B=\frac {(A+3)(A+2)(A+1)A} {24}$ and $B+3 \leq \frac {(A+3)^4} {24}$ since $A \geq 2,017,737,434,447,329$.

By Lemma ~\ref{induction}(2) \[s(H \wr S_4 \wr S_4) = {{B+3} \choose {B-1}}= {{B+3} \choose 4}=\frac {(B+3)(B+2)(B+1)B} {24} \leq \frac {(B+3)^4} {24} \leq \frac {(\frac {(A+3)^4} {24})^4} {24}. \]

Thus we have that \[rs(H \wr S_4 \wr S_4) \leq \frac {\log_2(A+3)} n- \frac {\log_2 24} {4n} - \frac {\log_2 24} {16n}.\]

By a similar argument, we will get that \[rs(H \wr S_4 \wr S_4 \wr S_4 \wr S_4) \leq \frac {\log_2(A+3)} n- \frac {\log_2 24} {4n} - \frac {\log_2 24} {16n} - \frac {\log_2 24} {64n} - \frac {\log_2 24} {256n}.\]

Let $m=\deg(P_1)$. By Lemma ~\ref{induction}(1), we have \[rs(H \wr P_1 \cdots \wr P_j) \geq \frac {\log_2 A} n- \frac {\log_2 |P_1|} {mn} - \frac {\log_2 \alpha} {mn}.\]

It suffices to show that \[\frac {\log_2 24} 4 + \frac {\log_2 24} {16} + \frac {\log_2 24} {64} + \frac {\log_2 24} {256} - \frac {\log_2 |P_1|} m - \frac {\log_2 \alpha} m \geq  \log_2 \frac {A+3} A.\]

Since $A \geq 2,017,737,434,447,329$. $\log_2 \frac {A+3} A \leq 0.0000000000000023$.

It suffices to show that \[\frac {\log_2 |P_1|} m + \frac {\log_2 \alpha} m  \leq 1.522350830317569088\]

We may use Theorem ~\ref{thm1}(2) and Table 3 to estimate $|P_1|$ and it is easy to check that the previous inequality is satisfied.
\end{proof}

Let $G_0=M_{24} \wr M_{12}$ and $G_k= G_0 \wr S_4 \wr \cdots \wr S_4$.

Thus $a_0=s(G_0)=2,017,737,434,447,329$, $s_1=s(G_1)= {s_0+3 \choose 4}$, $s_2=s(G_2)= {{s_1+3} \choose 4}$ and $s_k=s(G_k)={{s_{k-1}+3} \choose 4}$ by Lemma ~\ref{induction}(2).

We know that $rs(G_k)= \frac {\log_2 s(G_k)} {24 \cdot 12 \cdot 4^k}=\frac {\log_2 s_k} {24 \cdot 12 \cdot 4^k}=a_k$.

\begin{theorem} \label{thm4}
  Let $G$ be a permutation group of degree $n$ where $G$ does not contain any alternating group $\Alt(k), k>4$ as a composition factor. Let $s(G)$ denote the number of set-orbits of $G$. Then we have \[\inf(\frac {\log_2 s(G)} n)= \lim_{k \mapsto \infty} \frac {\log_2 s(M_{24} \wr M_{12} \wr S_4 \wr \cdots \wr S_4)} {24 \cdot 12 \cdot 4^k}=\lim_{k \mapsto \infty} a_k.\]
\end{theorem}
\begin{proof}
Let $M=\lim_{k \mapsto \infty} a_k$. Then clearly $M < a_2$. Suppose $G$ is primitive, then $rs(G) \geq a_2$ by Lemma ~\ref{lem4}.

Now we may assume $G$ is not primitive.

By Proposition ~\ref{prop1} we know that $G$ is induced from $M_{24}$.

By Proposition ~\ref{prop2} we know that $G$ is induced from $M_{24} \wr M_{12}$.

By Theorem ~\ref{thm3}, we have \[\inf(\frac {\log_2 s(G)} n)= \lim_{k \mapsto \infty} \frac {\log_2 s(M_{24} \wr M_{12} \wr S_4 \wr \cdots \wr S_4)} {24 \cdot 12 \cdot 4^k}.\]
\end{proof}

\textbf{Remark 1.} One can get a good estimate of the limit using the proof of Theorem ~\ref{thm3}.

Using the same notation, we have $G_0=M_{24} \wr M_{12}$ and we set $n_0=24 \cdot 12$. Thus by Theorem ~\ref{thm2}, we know that $M \geq 0.1712268716679245433$.

By the proof of Theorem ~\ref{thm3}, we have that $rs(G_0 \wr S_4 \wr S_4 \cdots) \leq \frac {\log_2(s_0+3)} {n_0}- \frac {\log_2 24} {4 n_0} - \frac {\log_2 24} {16 n_0} -\cdots$ where .

Clearly \[M=\lim_{k \mapsto \infty} a_k \leq \frac {\log_2(s_0+3)} n_0- \frac {\log_2 24} {n_0} \cdot( \frac 1 4+ \frac 1 {16} + \cdots) = \frac {\log_2(s_0+3)} {n_0}- \frac {\log_2 24} {3n_0} \approx 0.1712268716679245433.\] Taking into consideration the possible mistakes in the last digit, the following bound is guaranteed.
\[0.1712268716679245432 < M < 0.1712268716679245434.\]

\textbf{Remark 2.}
The following results are ~\cite[Theorem 2.6]{Keller} and ~\cite[Corollary 2.7]{Keller}.

\begin{theorem}
There is a number $C$ such that the following holds. Let $G$ be a finite group, $p$ be a prime not dividing $|G|$, and let $V$ be a finite faithful, irreducible $\FF_p G$-module. If $p > C$, then $k(GV) \geq 2 \sqrt {p - 1}$.
\end{theorem}

\begin{theorem}
Let $C$ be the constant occurring in the previous theorem. Let $G$ be a finite group. Suppose that $p$ is a prime dividing $|G|$ and that $p>C$. Then
\[k(G) \geq 2 \sqrt{p-1}.\]
\end{theorem}

It was mentioned in ~\cite{Keller} that the value $C$ one could get from the proof would be extremely large. But even finding such a bad value for $C$ seems to be quite a difficult task, given that the $C$ also depends on the unspecified constants in ~\cite{BAPYBER} and in ~\cite[Theorem 3·5(a)]{Keller1}.

The result of this paper would be helpful in finding an estimate of the constant $C$ in that paper since it provides the best possible estimate for the result in ~\cite{BAPYBER}.


\begin{table}[ht]
\caption{Partitions of $12$ and Number of orbits of $M_{24}$} 
\centering 
\begin{tabular}{c c c} 
\hline
\hline 
Number of blocks & Type & Number of orbits \\[0.5ex] 
\hline 
12 & (1,1,1,1,1,1,1,1,1,1,1,1) & 5040  \\ 
11 & (2,1,1,1,1,1,1,1,1,1,1) & 2520 \\
10 & (3,1,1,1,1,1,1,1,1,1) & 840 \\
10 & (2,2,1,1,1,1,1,1,1,1) & 1260 \\
9 & (4,1,1,1,1,1,1,1,1) & 210 \\
9 & (3,2,1,1,1,1,1,1,1) & 420 \\
9 & (2,2,2,1,1,1,1,1,1) & 630 \\
8 & (5,1,1,1,1,1,1,1) & 42 \\
8 & (4,2,1,1,1,1,1,1) & 105 \\
8 & (3,3,1,1,1,1,1,1) & 140 \\
8 & (3,2,2,1,1,1,1,1) & 210 \\
8 & (2,2,2,2,1,1,1,1) & 318 \\
7 & (6,1,1,1,1,1,1) & 7 \\
7 & (5,2,1,1,1,1,1) & 21 \\
7 & (4,2,2,1,1,1,1) & 54 \\
7 & (4,3,1,1,1,1,1) & 35 \\
7 & (3,3,2,1,1,1,1) & 70 \\
7 & (3,2,2,2,1,1,1) & 108 \\
7 & (2,2,2,2,2,1,1) & 165 \\
6 & (7,1,1,1,1,1) & 1 \\
6 & (6,2,1,1,1,1) & 4 \\
6 & (5,3,1,1,1,1) & 7 \\
6 & (5,2,2,1,1,1) & 12 \\
6 & (4,4,1,1,1,1) & 11 \\
6 & (4,3,2,1,1,1) & 19 \\
6 & (4,2,2,2,1,1) & 30 \\
6 & (3,3,3,1,1,1) & 24 \\
6 & (3,3,2,2,1,1) & 38 \\
6 & (3,2,2,2,2,1) & 59 \\
6 & (2,2,2,2,2,2) & 93 \\
5 & (8,1,1,1,1) & 1 \\
5 & (7,2,1,1,1) & 1 \\
5 & (6,3,1,1,1) & 2 \\
5 & (6,2,2,1,1) & 2 \\
5 & (5,4,1,1,1) & 4 \\
5 & (5,3,2,1,1) & 5 \\
5 & (5,2,2,2,1) & 8 \\
5 & (4,4,2,1,1) & 7 \\
5 & (4,3,3,1,1) & 8 \\
5 & (4,3,2,2,1) & 12 \\
5 & (4,2,2,2,2) & 20 \\
5 & (3,3,3,2,1) & 15 \\
5 & (3,3,2,2,2) & 23 \\
[1ex] 
\hline 
\end{tabular}
\label{table:par24nu1} 
\end{table}

\begin{table}[ht]
\caption{Partitions of $12$ and Number of orbits of $M_{12}$, continued} 
\centering 
\begin{tabular}{c c c} 
\hline
\hline 
Number of blocks & Type & Number of orbits \\[0.5ex] 
\hline 
4 & (9,1,1,1) & 1 \\
4 & (8,2,1,1) & 1 \\
4 & (7,3,1,1) & 1 \\
4 & (6,4,1,1) & 2 \\
4 & (5,5,1,1) & 3 \\
4 & (7,2,2,1) & 1 \\
4 & (6,3,2,1) & 2 \\
4 & (6,2,2,2) & 3 \\
4 & (5,4,2,1) & 3 \\
4 & (5,3,3,1) & 3 \\
4 & (5,3,2,2) & 4 \\
4 & (4,4,3,1) & 4 \\
4 & (4,3,3,2) & 6 \\
4 & (4,4,2,2) & 6 \\
4 & (3,3,3,3) & 8 \\
3 & (10,1,1) & 1 \\
3 & (9,2,1) & 1 \\
3 & (8,3,1) & 1 \\
3 & (8,2,2) & 1 \\
3 & (7,4,1) & 1 \\
3 & (7,3,2) & 1 \\
3 & (6,5,1) & 2 \\
3 & (6,4,2) & 2 \\
3 & (6,3,3) & 2 \\
3 & (5,5,2) & 2 \\
3 & (5,4,3) & 2 \\
3 & (4,4,4) & 3 \\
2 & (11,1) & 1 \\
2 & (10,2) & 1 \\
2 & (9,3) & 1 \\
2 & (8,4) & 1 \\
2 & (7,5) & 1 \\
2 & (6,6) & 2 \\
1 & (12) & 1 \\
[1ex] 
\hline 
\end{tabular}
\label{table:par24nu2} 
\end{table}

\begin{table}[ht]
\caption{Maximum Order of Primitive Groups not containing $A_n$} 
\centering 
\begin{tabular}{c c c} 
\hline
\hline 
Degree & Maximum Order & Second Largest Order\\[0.5ex] 
\hline 
5 & 20 \\
6 & 120\\
7 & 168\\
8 & 1344\\
9 & 1512\\
10 & 1440\\
11 & 7920\\
12 & 95040  & 7920\\
13 & 5616\\
14 & 2184\\
15 & 20160\\
16 & 322560\\
17 & 16320\\
18 & 4896\\
19 & 342\\
20 & 6840\\
21 & 120960\\
22 & 887040\\
23 & 10200960\\
24 & 244823040\\
25 & 28800\\
26 & 31200\\
27 & 303264\\
28 & 1451520\\
29 & 812\\
30 & 24360\\
31 & 9999360\\
32 & 319979520\\
33 & 163680\\
34 & N/A\\
35 & 40320\\
36 & 1451520\\
37 & 1332\\
38 & 50616\\
[1ex] 
\hline 
\end{tabular}
\label{table:primitivesmall} 
\end{table}


\end{document}